\documentclass[10pt]{article} 
\usepackage{amsmath,amssymb,latexsym,amsthm,enumerate,amscd}
\usepackage[affil-it]{authblk}
\usepackage[abbrev]{amsrefs} 
\usepackage{graphicx}
\usepackage{cite}
\usepackage{ textcomp }

\usepackage[titletoc]{appendix}
\numberwithin{equation}{section}

\usepackage[nocolor]{sseq}
\usepackage[textwidth=1.2in, textsize=small]{todonotes}
\usepackage{tikz, tikz-cd}
\usepackage{titlesec}
\usepackage[all]{xy}
\usetikzlibrary{arrows}
\usetikzlibrary{matrix}
\usetikzlibrary{shapes}
\usetikzlibrary{calc}
\usepgflibrary{shapes.geometric}

\theoremstyle{plain} 
\newtheorem{Theorem}{Theorem}[section]
\newtheorem{conjecture}[Theorem]{Conjecture} 
\newtheorem{Proposition}[Theorem]{Proposition} 
\newtheorem{Corollary}[Theorem]{Corollary} 
\newtheorem{Lemma}[Theorem]{Lemma}

\theoremstyle{remark}

\theoremstyle{definition} 
\newtheorem{Question}[Theorem]{Question}
\newtheorem{Definition}[Theorem]{Definition}

\newtheorem*{Acknowledgements}{Acknowledgements}

\newtheorem*{Strong Atiyah conjecture}{Strong Atiyah conjecture}
\newtheorem*{Strong Atiyah conjecture(Algebraic Version)}{Strong Atiyah conjecture (Algebraic Version)}
\newtheorem*{Main Theorem}{Main Theorem}

\usepackage{ifthen}

\newcommand{\showcomments}{yes}

\newsavebox{\commentbox}
{\ifthenelse{\equal{\showcomments}{yes}}%
{\footnotemark
    \begin{lrbox}{\commentbox}
    \begin{minipage}[t]{1.25in}\raggedright\sffamily\upshape\tiny
    \footnotemark[\arabic{footnote}]}
{\begin{lrbox}{\commentbox}}}%
{\ifthenelse{\equal{\showcomments}{yes}}%
{\end{minipage}\end{lrbox}\marginpar{\usebox{\commentbox}}}
{\end{lrbox}}}

\DeclareMathOperator{\lcm}{lcm} 
\title{Cocompact cubulations of mixed 3-manifolds}
\author{Joseph Tidmore}
\affil{
Department of Mathematical Sciences\\
University of Wisconsin-Milwaukee\\
Milwaukee, WI 53201, USA\\
email: jtidmore@uwm.edu}
\date{}

\begin{document}

\maketitle

\begin{abstract}
In this paper, we complete the classification of which compact 3-manifolds have a virtually compact special fundamental group by addressing the case of mixed 3-manifolds. A compact aspherical 3-manifold $M$ is mixed if its JSJ decomposition has at least one JSJ torus and at least one hyperbolic block. We show $\pi_1M$ is virtually compact special iff $M$ is chargeless, i.e. each interior Seifert fibered block has a trivial Euler number relative to the fibers of adjacent blocks.
\end{abstract}

\section{Introduction} \label{introduction}

\subsection{Statement of Main Result}

The special cube complexes of Haglund-Wise \cite{hw08} play a key role in the proof of the virtual Haken and virtual fibering conjectures. An important step in proving these conjectures is showing that the fundamental groups of hyperbolic 3-manifolds are virtually compact special i.e., virtually the fundamental group of a compact special cube complex, proved by Wise \cite{w12a} and Agol \cite{ago13} in the cusped hyperbolic and closed cases respectively. The main goal of this paper is to answer Question 9.4 of Aschenbrenner, Friedl, and Wilton in \cite{afw13} which we state below.

\begin{Question} \label{mainquestion}

Let $M$ be a compact, connected, aspherical 3-manifold. For which $M$ is $\pi_1M$ virtually compact special?

\end{Question}

For a geometric manifold $M$ which is not hyperbolic, then $\pi_1M$ is virtually compact special if and only if $M$ admits an $\mathbb{E}^3$, $\mathbb{H}^2 \times \mathbb{R}$, $S^2 \times \mathbb{R}$, or $S^3$ geometry by an observation of Hagen-Przytycki \cite{hp15}. The main result of Hagen-Przytycki \cite{hp15} was answering Question \ref{mainquestion} for graph manifolds, which left only the case when $M$ is a \emph{mixed manifold} unresolved, with mixed defined as follows: Let $M$ be a compact, connected, orientable, irreducible 3-manifold with $\partial{M}$ either empty or a disjoint union of tori. Then $M$ can be cut along tori called \emph{JSJ tori} so that each component is a hyperbolic 3-manifold or a Seifert fibered space. The 3-manifold $M$ is \emph{mixed} if this decomposition has at least one hyperbolic component and at least one JSJ torus. Each component of this \emph{JSJ decomposition} is a \emph{block}. Przytycki-Wise proved mixed manifold groups are virtually special in \cite{pw13} but did not address the issue of compactness. In this paper, we completely answer Question \ref{mainquestion} by studying the mixed manifold case.

The answer to Question 1.1 for mixed manifolds is similar to Hagen-Przytycki's answer for graph manifolds. They showed the obstruction to virtually compact special for nongeometric graph manifold groups is the charge. For a Seifert fibered block $B$ of a mixed manifold $M$ which is \emph{interior}, meaning $B$ does contain a boundary torus and is not adjacent to a hyperbolic block, the charge of $B$ is its Euler number relatively to the $S^1$-fibers of the adjacent blocks. The 3-manifold $M$ is \emph{chargeless} if all its interior Seifert fibered blocks are chargeless.

\begin{Main Theorem}
Let $M$ be a mixed manifold. The following are equivalent:
\begin{enumerate}
\item \label{Mchargeless} $M$ is chargeless.
\item \label{Mcocompact} $\pi_1M$ is virtually the fundamental group of a compact nonpositively curved cube complex.
\item \label{Mspecial} $\pi_1M$ is virtually compact special.
\end{enumerate}
\end{Main Theorem}

Przytycki-Wise in \cite{pw13} demonstrated that virtually special cubulations of the hyperbolic blocks and maximal graph manifold components could be combined using relatively hyperbolic techniques of Hruska-Wise \cite{hw14} to produce virtually special cubulations of mixed manifold groups (without addressing the issue of cocompactness). For chargeless mixed manifolds, we follow a similar strategy taking extra care to assure we preserve cocompactness. We combine the virtually compact special cubulations of Hagen-Przytycki \cite{hp15} for the chargeless graph manifold components with a more tightly constrained variation of Wise's \cite{w12a} virtually compact special cubulations of the hyperbolic blocks to produce virtually compact special cubulations for chargeless mixed manifold groups.

Showing that the fundamental group of a 3-manifold is virtually compact special has a number of consequences. Schreve in \cite{sch14} proved that any virtually compact special group $G$ satisfies the Strong Atiyah Conjecture. For a torsion free group G, the Strong Atiyah Conjecture states that for any space $\widetilde{X}$ with a properly discontinuous and cocompact $G$-action, the $L^2$-Betti numbers of $\widetilde{X}$ are integers.


\begin{Corollary}
Let $M$ be a chargeless mixed manifold. Then $\pi_1M$ satisfies the Strong Atiyah Conjecture.
\end{Corollary}


Niblo and Reeves in \cite{nr97} showed that cocompactly cubulated groups are biautomatic.

\begin{Corollary} \label{biautomatic}
Let $M$ be a chargeless mixed manifold. Then $\pi_1M$ is biautomatic.
\end{Corollary}

Corollary \ref{biautomatic} could also be derived from the main result of Rebecchi's thesis \cite{reb01} since $\pi_1M$ is hyperbolic relative to chargeless graph manifold groups which are biautomatic by Hagen-Przytycki \cite{hp15} and Niblo-Reeves \cite{nr97}.

The fundamental group of a mixed manifold $M$ has a natural relatively hyperbolic structure described in Section \ref{review}. Aschenbrenner, Friedl, and Wilton make the following conjecture for fully relatively quasiconvex subgroups (definition \ref{fullyrq}) of $\pi_1M$:

\begin{conjecture}[Conjecture 9.3 of \cite{afw13}] \label{fullyrelqc}

Let $M$ be a mixed manifold with $\pi_1M$ equipped with its natural relatively hyperbolic structure. If $H$ is a fully relatively quasiconvex subgroup of $\pi_1M$ then $H$ is a virtual retract. In particular, $H$ is separable.

\end{conjecture}

Theorem 5.8 of Chesebro, DeBlois, and Wilton \cite{cdw12} states any fully relatively quasiconvex subgroup of a relatively hyperbolic, virtually compact special group is a virtual retract. Thus we can partially answer this conjecture.

\begin{Corollary} \label{fully}
Let $M$ be a chargeless mixed manifold. Any fully relatively quasiconvex subgroup of $\pi_1M$ is a virtual retract and, in particular, is separable.
\end{Corollary}


Combining the main theorem with previously known results allows us to completely answer Question 1.1.

\begin{Theorem} \label{classification}

Let $M$ be a compact, aspherical 3-manifold whose boundary is empty or a disjoint union of tori. Then $\pi_1M$ is virtually compact special iff either of the following holds.
\begin{enumerate}
\item $M$ is geometric and its interior admits one of the following five geometries: $\mathbb{H}^3$, $\mathbb{E}^3$, $\mathbb{H}^2 \times \mathbb{R}$, $S^2 \times \mathbb{R}$, or $S^3$.

\item $M$ is nongeometric and chargeless.

\end{enumerate}

\end{Theorem}

\subsection{Outline for the proof of the main theorem}

In the main theorem, the implication $(\ref{Mspecial}) \implies (\ref{Mcocompact})$ is obvious. That $(\ref{Mcocompact}) \implies (\ref{Mchargeless})$ will be an application of Hagen-Przytycki \cite{hp15} and Theorem 7.12 of Hruska-Wise \cite{hw14}. Most of the work in this paper is proving $(\ref{Mchargeless}) \implies (\ref{Mspecial})$. Our strategy is to construct a collection of surfaces immersed in a chargeless mixed manifold $M$ and study the induced action of $\pi_1M$ on the dual $CAT(0)$ cube complex to show that $\pi_1M$ is virtually compact special. The construction of the dual cube complex, due to the Sageev, takes as input a collection of immersed codimension-$1$ surfaces in a 3-manifold $M$ and yields a $CAT(0)$ cube complex $\widetilde{X}$ dual this collection of surfaces together with an action of $\pi_1M$ on $\widetilde{X}$. Combinatorial features of the collection of immersed surfaces lead to various finiteness properties of the action of $\pi_1M$ on $\widetilde{X}$ such as proper, cocompact, special, etc.

Przytycki-Wise \cite{pw13} proved mixed manifold groups are virtually special (without addressing cocompactness) by combining collections of immersed surfaces due to Przytycki-Wise \cite{pw14} inducing virtually special cubulations of the graph manifold components and surfaces due to Wise \cite{w12a} inducing virtually compact special cubulations of the hyperbolic blocks to produce a certain collection of immersed surfaces in a mixed manifold. They then study the action on the dual cube complex using a theorem of Hruska-Wise \cite{hw14} to prove the action is proper with a virtually special quotient.


In general, the surfaces constructed by Przytycki-Wise do not provide a cocompact cubulation. One reason why is that a proper and cocompact actions requires additional constraints on how the surfaces intersect the JSJ and boundary tori. To see this, first consider a well-known example of Sageev's construction: Suppose we have a collection of closed curves in a torus $T$. The \emph{slope} of a closed curve $\gamma$ in $T$ is the commensurablility class of $\langle \gamma \rangle$ in $\pi_1T$. For a collection of closed curves with $n$ distinct slopes, the dual cube complex is $\mathbb{R}^n$ tessellated by $n$-cubes. Thus $\pi_1T = \mathbb{Z} \times \mathbb{Z}$ acts properly when $n \geq 2$ and acts cocompactly when $n \leq 2$. The action is proper and cocompact iff $n = 2$. Prztycki-Wise \cite{pw13} chose surfaces independently in the graph manifold components and hyperoblic blocks guaranteeing at least two slopes of curves in each JSJ torus $T$, but not exactly two since the slopes contributed by each block containing $T$ might not match.

For a chargeless mixed manifold $M$, we use in each graph manifold component the surfaces used by Hagen-Przytycki \cite{hp15} to obtain a virtually compact special cubulation of a chargeless graph manifold. The surfaces in the graph manifold components put a \emph{framing} on each JSJ torus contained in a graph manifold component, i.e. a choice of two slopes. We then add surfaces to the hyperbolic blocks whose boundary curves intersect the JSJ tori in the slopes that come from the framing. Our more tightly constriained variation of Wise's virtually compact special cubulation for cusped hyperbolic 3-manifold groups \cite{w12a} provides a collection of surfaces that induces a virtually compact special cubulation and that is true to any given framing of the boundary tori.

\begin{Theorem} \label{HyperbolicCube}

Let $N$ be a hyperbolic 3-manifold whose boundary is nonempty and a union of framed tori, $\partial{N} = T_1 \cup \cdots \cup T_k$. In each $T_i$ choose simple closed curves $C_i$ and $D_i$ whose slopes are those given by the framing. There is a finite collection $\mathcal{S}$ of surfaces properly immersed in $N$ which are geometrically finite and in general position so that $\pi_1N$ acts freely, properly, and cocompactly on the cube complex complex $\widetilde{X}$ dual to $\mathcal{S}$ and $X = \widetilde{X} / \pi_1N$ is virtually compact special. Further, if $H \subset X$ is an immersed hyperplane of $X$ and a conjugate of $\pi_1H \leq \pi_1X = \pi_1N$ intersects some $\pi_1T_i$ then that intersection lies in either $\pi_1C_i$ or $\pi_1D_i$.

\end{Theorem}

In Section \ref{setup} we define a \emph{frame efficient collection} of surfaces reducing the proof of the implication $(\ref{Mchargeless}) \implies (\ref{Mspecial})$ in the main theorem to two steps.

\begin{Proposition} \label{step1}

Let $M$ be a chargeless mixed manifold. Then $M$ admits a frame efficient collection of surfaces.

\end{Proposition}

\begin{Proposition} \label{step2}

Let $M$ be a mixed manifold admitting a frame efficient collection of surfaces $\mathbf{S}$. Then $\pi_1M$ is virtually compact special.

\end{Proposition}


$\mathbf{Organization.}$ Section \ref{setup} establishes notation, gives an overview of the cubulating techniques used in this paper, and gives a proof of the implication $(\ref{Mcocompact}) \implies (\ref{Mchargeless})$ of the main theorem. Section \ref{graph} describe a collection of surfaces immersed in a chargeless graph manifold constructed by Hagen-Przytycki \cite{hp15}, called an efficient collection. Section \ref{hyperbolic} constructs the surfaces we use in the hyperbolic blocks and proves Theorem \ref{HyperbolicCube}. Section \ref{cubulating} constructs a frame efficient collection of surfaces in a chargeless mixed manifold proving Proposition \ref{step1}. Section \ref{mainproof} proves Proposition \ref{step2}, completing the proof of the main theorem. Section \ref{finale} proves Theorem \ref{classification}, classifying virtually compact special 3-manifold groups.

\begin{Acknowledgements}

I would like to think my adviser Chris Hruska for all his help and advice throughout this paper. I would also like to think Boris Okun and Kevin Schreve for helpful clarifications on the Strong Atiyah conjecture.

\end{Acknowledgements}

\section{Background} \label{setup}

We need several tools to prove the main theorem. Section \ref{3mfld} provides some background in 3-manifold theory. Section \ref{efficient} defines a frame efficient collection of surfaces in a mixed manifold. Section \ref{review} describes a natural relatively hyperbolic structure of mixed manifold groups and a key result of \cite{hw14} for cubulating relatively hyperbolic groups.

\subsection{3-Manifold Background} \label{3mfld}

Here we describe two decompositions for mixed manifolds and present some background in 3-manifold theory. A good reference for many of the results here is \cite{afw13}.

\textbf{Modified JSJ decomposition.} We first describe the classical JSJ composition. Let $M$ be a compact connected oriented irreducible 3-manifold whose boundary is either empty or a disjoint union of tori. The 3-manifold $M$ has a unique, up to isotopy, minimal collection of incompressible tori which are not $\partial$-parallel called \emph{JSJ tori} such that when $M$ is cut open along these tori each component of the cut-open space, called a \emph{block} of $M$, is either atoroidal or admits a Seifert fibered structure.  The 3-manifold $M$ is a \emph{mixed manifold} if it has at least one JSJ torus and at least one atoroidal block. When $M$ has at least one JSJ torus, Thurston's hyperbolization tells us each atoroidal block of $M$ admits a hyperbolic structure. We refer to the blocks of our mixed manifold as hyperbolic blocks and Seifert fibered blocks as appropriate.

Let $M_0$ denote the space obtained by cutting $M$ along all the JSJ tori. Each JSJ torus $T \subset M$ is the preimage of two distinct components $T_1$ and $T_2$ of $\partial{M_0}$. If $B_1$ and $B_2$ are the blocks containing $T_1$ and $T_2$ respectively we say $B_1$ and $B_2$ each \emph{contain $T$}. We also say $B_1$ and $B_2$ are \emph{adjacent via $T$}. Note it is possible that $B_1 = B_2$ so a block can be adjacent to itself via some JSJ torus.

As in \cite{pw13} we modify the above the decomposition in a way that is useful for mixed manifolds. Suppose $T$ is either a JSJ or boundary torus of $M$ and is not contained in any Seifert fibered block. Choose a parallel copy of $T$ in $M$ and call it a JSJ torus also. The product region $T \times I$ bounded $T$ and this parallel copy has many possible Seifert fibered structures and we call it a \emph{thin} Seifert fibered block. Throughout this paper when we consider a mixed manifold we will refer to this modified decomposition as its JSJ decomposition.

\textbf{Transitional decomposition.} We also use another decomposition from \cite{pw13} for mixed manifolds. A JSJ torus $T$ of a mixed manifold $M$ is a \emph{transitional torus} if it is contained in at least one hyperbolic block. Cutting $M$ along the transitional tori gives us its \emph{transitional decomposition} with each component being either a hyperbolic block of $M$ or a graph manifold which we call a \emph{graph manifold cluster} of $M$. If a graph manifold cluster $N$ of $M$ consists of just a thin Seifert fibered block, $N$ is \emph{thin}. The other graph manifold clusters are \emph{thick}. We define adjacency for components of the transitional decomposition in an analogous way to that of blocks of the JSJ decomposition.  Note the transitional decomposition is bipartite in that a hyperbolic block is only adjacent to graph manifold clusters and vice versa.

\textbf{Elevation.} Let $\phi \colon N \to M$ be a map between manifolds of any dimension and $\widehat{M} \to M$ a covering space. A map $\widehat{\phi} \colon \widehat{N} \to \widehat{M}$, with $\widehat{N}$ a cover of $N$, is an \emph{elevation} of $\phi$ if $\widehat{\phi}$ covers $\phi$ and does not factor through any intermediate cover of $N$. 

The elevations of the JSJ, boundary, and transitional tori of $M$ to its universal cover $\widetilde{M}$ are \emph{JSJ}, \emph{boundary}, and \emph{transitional planes} respectively. Similarly, the elevations of the hyperbolic and Seifert fibered blocks are hyperbolic and Seifert fibered blocks of $\widetilde{M}$ respectively. The elevations of graph manifold clusters to $\widetilde{M}$ are graph manifold clusters of $\widetilde{M}$. This last case is a slight abuse of notation since these elevations are not compact.  In our notation, a graph manifold cluster of $\widetilde{M}$ is a connected subspace of $\widetilde{M}$ which covers a graph manifold cluster of $M$.

\textbf{Properties of Immersed Surfaces.} An connected, immersed (embedded) surface $\phi \colon S \to M$ in a 3-manifold $M$ is \emph{properly immersed (embedded)} if $\phi^{-1} (\partial M) = \partial S$. An immersed surface $\phi \colon S \to M$ which is not a $2$-sphere is \emph{immersed incompressible} if $\phi$ is $\pi_1$-injective and elevates to an embedding in $\widetilde{M}$.

\textbf{Pieces of Surfaces.} If $B$ is a hyperbolic (Seifert fibered) block of $M$, any restriction of $\phi$ to a component of $\phi^{-1}(B)$ is a \emph{hyperbolic (Seifert fibered) piece of $S$ in $B$}. For a graph manifold cluster $N$ of $M$, we call a component of $\phi^{-1}(N)$ a \emph{piece cluster}.


\textbf{Accidental parabolic.} Let $S \to M$ be a properly immersed surface and $C$ a closed curve in $S$. Suppose the image of $C$ in $M$ is freely homotopic in $M$ to a curve lying in some transitional torus $T$. Then $C$ is an \emph{accidental parabolic} if there is a homotopy of $S \to M$ so that both the following hold: The image of $C$ in $M$ lies in the interior of a hyperbolic block, and $C$ is not freely homotopic in $S$ to a curve that maps into $T$.

\textbf{Chargeless.} Rather than defining the notion of the charge of a Seifert fibered block, we instead discuss only the notion of chargeless blocks. This condition concerns \emph{interior} Seifert fibered blocks of $M$, i.e. those that neither contain a boundary torus of $M$ nor are adjacent to a hyperbolic block.

\begin{Definition}\label{chargeless}
Let $M$ be a compact, oriented, irreducible 3-manifold with at least one JSJ torus and whose boundary is either empty or a union of tori. An interior Seifert fibered block $B$ of $M$ is \emph{chargeless} if the following holds:

Let $T_1$, ... ,$T_k$ be the JSJ tori contained in $B$. For each $T_i$, let $B'_i$ denote the Seifert fibered block adjacent to $B$ via $T_i$ and $Z_i$ a circle in $T_i$ which is a fiber of $B'_i$. Let $[Z_i]$ denote class of $Z_i$ in $H_1(B;\mathbb{Z})$. $B$ is \emph{chargeless} if we can assign nonzero integers $n_1$, ... , $n_k$ so that

\[ \sum\limits_{i=1}^{k} n_{i} [Z_i] = 0 \text{ in } H_1(B;\mathbb{Z}). \]

We say $M$ is \emph{chargeless} if every interior Seifert fibered block of $M$ is chargeless.

\end{Definition}

We will later see the cycle $\sum\limits_{i=1}^{k} n_{i} [Z_i] = 0 \text{ in } H_1(B;\mathbb{Z})$ bounds an embedded horizontal surface in $B$. An analogous property holds for all hyperbolic blocks as a consequence of Theorem \ref{stop} proved later. In the latter case the cycle $\sum\limits_{i=1}^{k} n_{i} [Z_i] = 0 \text{ in } H_1(B;\mathbb{Z})$ bounds a geometrically finite surface. This plays a key role in the construction of a frame efficient collection of surfaces.

We need a few facts about hyperbolic 3-manifolds and Seifert fibered spaces.

\textbf{Geometrically Finite.} Suppose $G$ is a Kleinian group i.e., a discrete subgroup of $PSL(2, \mathbb{C}) = Isom^{+}(\mathbb{H}^3)$. Consider the spherical boundary $S^2_{\infty}$ of $\mathbb{H}^3$. For any $x \in \mathbb{H}^3$, the \emph{limit set of $G$}, denoted by $\Lambda_G$, is the set of all accumulation points of $Gx$ in $S^2_{\infty}$. Note this is independent of the choice of $x \in \mathbb{H}^3$. Let $C(\Lambda_G)$ denote the convex hull of $\Lambda_G$ in $\mathbb{H}^3 \cup S^2_{\infty}$. Then $G$ is \emph{geometrically finite} if there is an $\epsilon > 0$ so that $\mathcal{N}_{\epsilon} (C(\Lambda_G \cap \mathbb{H}^3) / G$ has finite volume. An immersed incompressible surface $S$ in a hyperbolic 3-manifold $N$ is \emph{geometrically finite} if $\pi_1S \leq \pi_1N \leq Isom^{+}(\mathbb{H}^3)$ is geometrically finite.

\textbf{Horizontal and Vertical.} A surface immersed in a Seifert fibered space is \emph{horizontal} if it only has transverse intersections with the $S^1$-fibers. It is \emph{vertical} if it a union of $S^1$-fibers. Hass showed in \cite{has84} that every immersed (embedded) surface in a Seifert fibered space is a homotopic to either a horizontal or vertical immersion (embedding). Rubinstein and Wang applied this to show that a surface immersed in a graph manifold can be homotoped so that each piece is either horizontal or vertical in the Seifert fibered block it maps into. (Lemma 3.3 of \cite{rw98}.) We will assume the piece clusters of any surface we consider to be homotoped into this form.

\subsection{Frame Efficient Collection} \label{DefineEfficient}

The goal of this subsection is to define the notion of a frame efficient collection of properly immersed surfaces. For this section, let $M$ be a mixed manifold and $\widetilde{M}$ its universal cover.

First, we describe how a collection of surfaces gives $\widetilde{M}$ the structure of a wallspace in the sense of Haglund-Paulin. Suppose $\mathbf{S}$ is a collection of properly immersed incompressible surfaces in a mixed manifold $M$. Let $\widetilde{\mathbf{S}}$ denote the collection of all elevations of surfaces in $\mathbf{S}$ to $\widetilde{M}$. Each $\widetilde{S} \in \mathbf{\widetilde{S}}$ is a \emph{wall} in $\widetilde{M}$ meaning that cutting $\widetilde{M}$ along $\widetilde{S}$ decomposes it into two halfspaces $U$ and $V$. Thus $\widetilde{\mathbf{S}}$ endows $\widetilde{M}$ with a Haglund-Paulin wallspace structure $( \widetilde{M}, \widetilde{\mathbf{S}} )$. (We follow the more flexible treatment in \cite{hw14} where $U \cap V$ can be nonempty.)

The $CAT(0)$ cube complex $\widetilde{X}$ dual to the wallspace $( \widetilde{M}, \widetilde{\mathbf{S}} )$ was first constructed by Sageev \cite{sag95}. The action of $\pi_1M$ on $\widetilde{M}$ preserves the wallspace structure inducing an action of $\pi_1M$ on the dual cube complex. We do not describe Sageev's construction here. For background see e.g. Hruska-Wise \cite{hw14} which gives a self-contained account similar to the treatment in this paper. We need to highlight some key properties. A \emph{midcube} of an $n$-cube $[ -1, 1 ]^n$ is a subspace obtained by restricting one of its coordinates to $0$. A \emph{hyperplane} $\widetilde{H}$ of $\widetilde{X}$ is a connected subspace which intersects each cube of $\widetilde{X}$ in either a midcube or the empty set. Each wall $\widetilde{S} \in \mathbf{\widetilde{S}}$ is associated to a unique hyperplane $\widetilde{H}$ of $\widetilde{X}$ and $\widetilde{H}$ has the property that $stab(\widetilde{S}) = stab(\widetilde{H})$ implying $\pi_1S \leq \pi_1M$ is a finite index subgroup of $stab(\widetilde{H})$.

The statement of Theorem \ref{HyperbolicCube} uses the notion of an \emph{immersed hyperplane} in a nonpositively curve cube complex $X$ so we define this as well. Given a hyperplane $\widetilde{H}$ in $\widetilde{X}$, the universal cover of $X$, with $K = stab(\widetilde{H}) \leq \pi_1X$, the induced map $H = \widetilde{H} / K \to X$ is an \emph{immersed hyperplane} of $X$. Note that $H \to X$ is a local isometry and hence $\pi_1$-injective.

A cube in a cube complex is \emph{maximal} if it is not a proper subset of another cube. If the dual $CAT(0)$ cube complex $\widetilde{X}$ is finite-dimensional and we consider a maximal collection of pairwise crossing (i.e., intersecting) walls then the collection of hyperplanes associated to those walls is a maximal collection of pairwise crossing hyperplanes. Further, each maximal cube of $\widetilde{X}$ corresponds to a unique maximal collection of pairwise crossing hyperplanes. In fact, they cross in that cube. Thus each maximal collection of pairwise crossing walls corresponds to the unique maximal cube of $\widetilde{X}$. A common strategy for proving a group acts cocompactly on a dual cube complex is to show there are finitely many orbits of collections of pairwise crossing walls.

To define a frame efficient collection of surfaces we need some terminology.

\textbf{Cut-surface.} An \emph{axis} for a nontrivial element $g \in \pi_1M$ acting on $\widetilde{M}$ is a copy of $\mathbb{R}$ in $\widetilde{M}$ on which $g$ acts by nontrivial translation. A \emph{cut-surface} for $g \in \pi_1M$ is an immersed incompressible surface $S \to M$ covered by $\widetilde{S} \subset \widetilde{M}$ such that there is an axis $\mathbb{R}$ for $g$ satisfying $\widetilde{S} \cap \mathbb{R} = \{0\}$, where the intersection is transverse.

The existence of cut-surfaces is important for proper actions. Suppose $\mathbf{S}$ is a collection of properly immersed incompressible surfaces containing a cut-surface for every nontrivial element of $\pi_1M$. If the dual cube complex of $\mathbf{S}$ is finite-dimensional, then $\pi_1M$ acts freely and properly on it. (See Theorem 5.5 of \cite{hw14}.)

\textbf{Strong Separation.} Finally, we need to define the \emph{Strong Separation} property for a collection of properly immersed surfaces $\mathbf{S}$, which is Definition 2.2 of \cite{pw13}. Equip $M$ with a Riemannian metric and lift it to $\widetilde{M}$. Let $\widetilde{\mathbf{S}}$ be the collection of all elevations of surfaces in $\mathbf{S}$ to $\widetilde{M}$. Then $\mathbf{S}$ satisfies the \emph{Strong Separation} property if there exists $D > 0$ so that the following hold:

\begin{enumerate}

\item Suppose $\widetilde{S}$, $\widetilde{S}' \in \mathbf{\widetilde{S}}$ both intersect a hyperbolic block $\widetilde{N}$. If there is no JSJ plane contained in $\widetilde{N}$ intersecting both $\widetilde{S}$ and $\widetilde{S}'$ and $\widetilde{S} \cap \widetilde{N}$ and $\widetilde{S}' \cap \widetilde{N}$ are distance $\geq D$ from each other then there is a surface in $\mathbf{\widetilde{S}}$ which separates $\widetilde{S}$ from $\widetilde{S}'$.

\item Suppose $\widetilde{S}$, $\widetilde{S}' \in \mathbf{\widetilde{S}}$ both intersect a graph manifold cluster $\widetilde{N}'$. If $\widetilde{S} \cap \widetilde{N}'$ and $\widetilde{S}' \cap \widetilde{N}'$ are distance $\geq D$ from each other, then there is a surface from $\mathbf{\widetilde{S}}$ which separates $\widetilde{S}$ from $\widetilde{S}'$.

\end{enumerate}

Note whether or not a collection of surfaces satisfies the Strong Separation property is independent of the Riemannian metric we chose.

\begin{Definition} \label{efficient}

Let $M$ be a mixed manifold and $\mathbf{S}$ a finite collection of properly immersed incompressible surfaces in $M$ which are in general position. Choose any Seifert fibration of the thin graph manifold clusters of $M$. Then $\mathbf{S}$ is a \emph{frame efficient collection} if all of the following hold:

\begin{enumerate}
\item \label{cut} Each nontrivial element of $\pi_1M$ has a cut-surface in $\mathbf{S}$.

\item \label{jsj} All JSJ tori belong to $\mathbf{S}$.

\item \label{vembedded} For any piece cluster $S_0 \subset S$, the map $S_0 \to N$ into a graph manifold cluster is a virtual embedding for all $S \in \mathbf{S}$.

\item \label{geometricallyfinite} Each hyperbolic piece of $S$ is geometrically finite for all $S \in \mathbf{S}$.

\item \label{strongseparation} The collection $\mathbf{S}$ satisfies the Strong Separation property.


\item \label{turbine} Two horizontal Seifert fibered pieces of a surface $S \in \mathbf{S}$ cannot be directly attached in the following sense: Suppose $B$ is a Seifert fibered block of $M$ and a piece of $S_0 \subset S$ is immersed horizontally in $B$. If $B'$ is a Seifert block adjacent to $B$ via a JSJ torus $T$, then each component of $S_0 \cap T$ is an $S^1$-fiber of $B'$.


\item \label{coveringembedding} Let $B$ be a Seifert fibered block. The images of pieces of surfaces immersed horizontally in $B$ do not intersect one another. Further, each piece immersed horizontally in $B$ maps into $B$ via the composition of a covering map between surfaces and an embedding into $B$.


\item \label{verticalaccidental} The accidental parabolics are all vertical in the following sense: Suppose $S \in \mathbf{S}$ contains an accidental parabolic $C \subset S$ freely homotopic into a transitional torus $T$. Let $B$ denote the Seifert fibered block adjacent to $T$. Then the image of $C$ is freely homotopic to a fiber of $B$.

\end{enumerate}
\end{Definition}

Przytycki-Wise showed a mixed manifold $M$ admits a collection of properly immersed incompressible surfaces satisfying Definition \ref{efficient}(1)-(5) (Theorem 2.1 of \cite{pw13}) and that $\pi_1M$ acts freely on the dual $CAT(0)$ cube complex of any such collection with a virtually special quotient.  Criteria (\ref{turbine}) and (\ref{coveringembedding}) are motivated by the collected of surfaces used by Hagen-Przytycki in \cite{hp15}, which they call an efficient collection. They had embedded horizontal pieces, but for technical reasons that emerge in Section \ref{cubulating} where we construct a frame efficient collection we need our weaker criterion (\ref{coveringembedding}). The significance of criterion (\ref{verticalaccidental}) emerges in the following discussion.

\textbf{Two slopes in each torus.} Implicit in Definition \ref{efficient} is that the surfaces of a frame efficient intersect each JSJ and boundary torus $T$ in a collection of closed curves with exactly two slopes. Definition \ref{efficient}(\ref{cut}) implies there are at least two slopes in order to provide cut-surfaces for the elements of these tori subgroups. Definition \ref{efficient}(\ref{coveringembedding}) implies there are at most two since, if $B$ is the Seifert fibered block containing $T$, the pieces of surfaces immersed horizontally in $B$ all intersect $T$ in curves of the same slope. The second slope is the slope of an $S^1$-fiber of $B$. When studying cocompactness of the action on the dual cube complex, accidental parabolics homotopic into $T$ behave similar to curves in $T$. Definition \ref{efficient}(8) is stronger than necessary since we only need to ensure the accidental parabolics do not add new slopes to $T$, but requiring that they are vertical in $B$ simplifies the proof of Proposition \ref{step2}.

Being aware of the necessary condition that surfaces intersects the JSJ and boundary tori in two slopes of curves is key to understanding our methods when we later construct a frame efficient collection, but we never use this condition explicitly when proving the main theorem. A full explanation of why this condition is necessary requires understanding how a subspace of $M$ can be associated to a convex subcomplex of the dual cube complex of a collection of surfaces, which Hruska-Wise describe in \cite{hw14}. Since we never need to explicitly use this condition, we do not give the full explanation here.



\textbf{Constructing a frame efficient collection.} To construct a frame efficient collection we modify the strategy of Przytycki-Wise \cite{pw13} for constructing a collection of surfaces satisfying Definition \ref{efficient}(1)-(5). They first chose surfaces immersed in the hyperbolic blocks and graph manifold clusters. These surfaces were not all properly immersed in $M$, since some of them had boundary components lying in transtional tori. To extend these surfaces to be properly immersed in $M$, they first added extra surfaces in the hyperbolic blocks and graph manifold clusters with the slopes of their boundary curves chosen so that they match the slopes of boundary curves from surfaces in adjacent blocks. They then used  these extra surfaces to ``cap off'' the boundary curves of surfaces in adjacent blocks and obtain surfaces properly immersed in their mixed manifold. The result of this strategy is a virtually special cubulation but not, in general, a cocompact cubulation since we could have as many as four slopes in a JSJ or boundary torus.

To construct a frame efficient collection we first add to each graph manifold cluster the efficient collection of surfaces used in Hagen-Przytycki \cite{hp15}, described in Section \ref{graph}. These surfaces add exactly two slopes to each transitional torus. The boundary curves from surfaces in the efficient collections equip each hyperbolic block $N$ with a \emph{framing} in the sense defined below.

\begin{Definition} \label{frame}

Let $N$ be a compact 3-manifold whose boundary is a nonempty union of tori $\partial{N} = T_1 \cdots T_k$. A \emph{framing} of $N$ is a choice in each $T_i$ of two nonhomotopic simple closed curves $C_i$ and $D_i$. (Alternatively, we could choose a pair of slopes in each $T_i$.) If a framing for $N$ is chosen, then $N$ is \emph{framed}. 

A collection of properly immersed surfaces $\mathcal{S}$ in $N$ is \emph{true to a $\{ C_i , D_i \}$-framing} (or \emph{true to the framing} if the framing has already been specified) if, for each $S \in \mathcal{S}$ and each $T_i$, every component of $\partial{S}$ immersed in $T_i$ has the same slope as either $C_i$ or $D_i$.

\end{Definition}

In Section \ref{hyperbolic}, we prove that a hyperbolic block $N$ with any framing admits a collection of geometrically finite surfaces that is true to the framing and that provides a cut-surface for every nontrivial element of $\pi_1N$. To prove Proposition \ref{step1}, we construct in Section \ref{cubulating} a frame efficient collection by attaching surfaces from the efficient collections in the graph manifold clusters to surfaces in the hyperbolic blocks that are true to the framings induced by the efficient collections in the graph manifold clusters.


\subsection{Relatively Hyperbolic Groups and Cube Complexes} \label{review}

In this section we describe a natural relatively hyperbolic structure on the fundamental group of a mixed manifold $M$ and its role in finding cocompact cubulations. We also prove Theorem \ref{CocompactImpliesChargeless} which proves the implication in the main theorem that $\pi_1M$ being virtually cocompactly cubulated implies $M$ is chargeless.

Gromov originally introduced the notion of a relatively hyperbolic group in \cite{gro87}.  The definition we give here is due to Bowditch \cite{bow12}. For finitely generated groups, it is equivalent to Gromov's definition. (See \cite{hru10} for more on the various definitions for relatively hyperbolic.)

\begin{Definition}[Definition 2 of \cite{bow12}] \label{relhyp}

Suppose $G$ is a group acting on a connected hyperbolic graph $\Gamma$. Suppose the following all hold.
\begin{enumerate}

\item $\Gamma$ is $\delta$-hyperbolic.
\item For any positive integer $n$, each edge of $\Gamma$ lies in only finitely many circuits of length $n$, a circuit being a closed path which does not repeat any vertices.
\item There are only finitely many $G$-orbits of edges and each edge stabilizer is finite.
\item Each vertex stabilizer is finitely generated.
\end{enumerate}

Let $\mathbb{P}$ be a collection of subgroups consisting on one representative from each conjugacy class of infinite vertex stabilizers. We say $G$ is \emph{hyperbolic relative to $\mathbb{P}$}. The subgroups of $\mathbb{P}$ and their conjugates are the \emph{peripheral subgroups} of $G$. (In other words, the infinite vertex stabilizer of $G$ are its peripheral subgroups.) We call $\Gamma$ a \emph{$(G$, $\mathbb{P})$-graph}.


\end{Definition}

We will not appeal directly to the definition of relatively hyperbolic, but we make use of a natural relatively hyperbolic structure on mixed manifold groups first described by Dru{\c{t}}u-Sapir in \cite{ds05}. To prove mixed manifolds have this relatively hyperbolic structure, Dru{\c{t}}u-Sapir used highly intricate techniques combining their results with a result of Kapovich-Leeb \cite{kl95}. Kapovich-Leeb showed the asymptotic cone of a mixed manifold group is tree-graded and Dru{\c{t}}u-Sapir showed asymptotically tree graded groups are relatively hyperbolic. For an elementary proof of the theorem below using the $(G , \mathbb{P})$-graph definition of relatively hyperbolic, see \cite{bw13}.

\begin{Theorem}[Dru{\c{t}}u-Sapir] \label{relhyp}

Let $M$ be a mixed manifold and let $N_1 , ... , N_k$ denote the graph manifold clusters of $M$. For each $N_i$ choose a conjugate $P_i$ of $\pi_1N_i$ sitting inside $\pi_1M$. The group $\pi_1M$ is hyperbolic relative to $\{ P_i \}$.

\end{Theorem}

We also describe the notion of a relatively quasiconvex subgroup. Introduced by Dahmani in \cite{dah03}, relative quasiconvexity is also a rich property with many equivalent definitions. We use a definition in the hyperbolic graph setting due to Mart{\'{\i}}nez-Pedroza and Wise \cite{mw11}

\begin{Definition}

Let $G$ be hyperbolic relative to subgroups $\{ \mathbb{P} \}$ and $\Gamma$ a $(G$, $\{ \mathbb{P} \})$-graph. Suppose $H \leq G$ is a subgroup. Then $H$ is \emph{relatively quasiconvex} in $G$ if there is a quasi-isometrically embedded subgraph $K$ of $\Gamma$ which is $H$-invariant and has finitely many $H$-orbits of edges. 


\end{Definition}

Surfaces in a frame efficient collection correspond to relatively quasiconvex subgroups by the following application of results of Hruska \cite{hru10} and Bigdely-Wise \cite{bw13}:

\begin{Proposition}[Przytycki-Wise in \cite{pw13}] \label{relhypsurface}

Let $M$ be a mixed manifold and $S \to M$ a properly immersed incompressible surface. Suppose each piece of $S$ in a hyperbolic block is geometrically finite. Then $\pi_1S$ maps into $\pi_1M$ as a relatively quasiconvex subgroup.

\end{Proposition}

\begin{proof}

If $N$ is a hyperbolic block and $S_0 \to N$ is a geometrically finite piece of $S$ in $N$ then $\pi_1S_0$ is relatively quasiconvex in $\pi_1N$ by Corollary 1.3 of \cite{hru10}. It then follows from Theorem 4.17 of \cite{bw13} that $\pi_1S$ is relatively hyperbolic in $\pi_1M$.

\end{proof}

Corollary \ref{fully} deals with the notion of fully relatively quasiconvex subgroups, a notion also introduced by Dahmani \cite{dah03}, so we give this definition as well.

\begin{Definition} \label{fullyrq}

Let $G$ be a relatively hyperbolic group. A relatively quasiconvex subgroup $H$ of $G$ is \emph{fully relatively quasiconvex} if each intersection of $H$ with a peripheral subgroup of $G$ is either finite or finite index.

\end{Definition}

To motivate how we use relative hyperbolicity and relative quasiconvexity and provide some necessary background, let us consider the word hyperbolic case. Let $G$ be a word hyperbolic group acting on a wallspace $( Y , \mathcal{W} )$. Sageev proved in \cite{sag97} that if there are finitely many $G$-orbits of hyperplanes and each hyperplane stabilizer is quasiconvex in $G$, then $G$ acts cocompactly on $\widetilde{X}$. Note Sageev worked in a different setting where $G$ acts on its Cayley graph and instead of walls we have codimension-$1$ subgroups with each codimension-$1$ subgroup $H$ associated to an $H$-almost invariant set. There are two steps to his proof. We need both of these facts for our proof of Proposition \ref{step2} so we state them. 

The first step, Lemma \ref{HyperbolicCenter}, appears implicitly in \cite{sag97} where Sageev deduced it from results in \cite{gmrs98}. Our version is a slight modification of Lemma 7.3 of \cite{hw14} tailored to the wallspace setting.

\begin{Lemma}[Sageev] \label{HyperbolicCenter}

Suppose a group $G$ acts properly and cocompactly by isometries on a wallspace $( Y , \mathcal{W} )$ with $Y$ a $\delta$-hyperbolic space. Suppose there are finitely many $G$-orbits of walls and each wall stabilizer is $\kappa$-quasiconvex in $G$. Then for any $D \geq 0$ there is a constant $L = L( D , \delta , \kappa )$ so that the following holds: Let $\mathcal{V} \subset \mathcal{W}$ be a collection of pairwise $D$-close walls in $Y$. There is a point $y_0 \in Y$ which is distance $\leq L$ from each wall of $\mathcal{V}$.

In particular, there are finitely many orbits of pairwise crossing walls.

\end{Lemma}

Whenever the conclusion of Lemma \ref{HyperbolicCenter} holds we say $y_0$ is an \emph{$L$-center of $\mathcal{V}$}. The second step is Lemma \ref{center}.

\begin{Lemma} \label{center}

Let $G$ be a finitely generated group acting on a wallspace $( Y , \mathcal{W} )$. Suppose $G$ acts cocompactly on a subspace $Z \subset Y$. Suppose for any $D > 0$ there exists a constant $L = L(D)$ with the following property: If $\mathcal{V} \subset \mathcal{W}$ is a collection of pairwise $D$-close walls in $Y$ then there is a point $z_0 \in Z$ so that each wall in $\mathcal{V}$ is distance $\leq L$ from $z_0$. Then there are finitely many orbits of collections of pairwise $D$-close walls. In particular, $G$ acts cocompactly on the dual $CAT(0)$ cube complex $\widetilde{X}$ of $( Y, \mathcal{W})$.

\end{Lemma}

\begin{proof}

For each collection of pairwise $D$-close walls, choose an $L$-center in $Z$ for that collection. Since $G$ acts cocompactly on $Z$ we can assume, possibly enlarging $L$, that we have finitely many $G$-orbits of $L$-centers. A closed ball of radius $D$ can only intersect finitely many walls. Since there are finitely many orbits of $D$-centers, this puts an upper bound on the size of any collection of pairwise $D$-close walls. In particular, there is an upper bound on the size of any collection of pairwise crossing walls. Therefore $\widetilde{X}$ is finite dimensional and each cube lies in a maximal cube.

Some of the points we chose might be a center for more than one collection, but they can each only be a center for finitely many collections. Since there are finitely many $G$-orbits of centers this implies $\mathcal{W}$ contains finitely many collections of pairwise $D$-close walls. In particular, there are finitely many collections of pairwise crossing walls. Therefore $G$ acts cocompactly on $\widetilde{X}$.

\end{proof}

When applying Lemma \ref{center}, we will consider the situation where $G$ is a peripheral subgroup of a mixed manifold group and our wallspace is a $Z$-wallspace, defined below:

\begin{Definition} \label{hemi}

Let $( Y , \mathcal{W} )$ be a wallspace with $Y$ a metric space. Let $Z \subset Y$ be a subspace with $diam(Z) = \infty$. Let $\mathcal{W}_Z \subset \mathcal{W}$ consists of walls all $W \in \mathcal{W}$ with the following property: There is $r > 0$ so that if $W$ decomposes $Y$ into halfspaces $U$ and $V$ then $diam( U \cap \mathcal{N}_r(Z)) = diam( V \cap \mathcal{N}_r(Z)) = \infty$.

Then $(Y , \mathcal{W}_Z)$ is the \emph{$Z$-wallspace of $( Y , \mathcal{W}_Z )$}.

\end{Definition}

Let $C(Z)$ be the dual $CAT(0)$ cube complex of $(Y , \mathcal{W}_Z)$. If $\widetilde{X}$ is the dual cube complex of $(Y , \mathcal{W})$, then there is a canonical embedding of $C(Z)$ as a convex subcomplex of $\widetilde{X}$. (See sections 3.4 and 7.2 of Hruska-Wise \cite{hw14}.) We call $C(P)$ the \emph{convex subcomplex associated to $P$}.

The following is our main tool for verifying cocompactness.

\begin{Theorem}[Theorem 7.12 \cite{hw14}]\label{relcpt}

Let $( Y , \mathcal{W} )$ be a wallspace such that $Y$ is also a length space. Suppose a group $G$ acts properly and cocompactly by isometries on $Y$ preserving its wallspace structure. Suppose the action on $\mathcal{W}$ has finitely many $G$-orbits of walls. Suppose $G$ is hyperbolic relative to a finite collection of subgroups $ \{ P_i \}$. Suppose for each $W \in \mathcal{W}$ that $H = Stab(W)$ acts cocompactly on $W$ and $H$ is relatively quasiconvex in $G$. For each peripheral subgroup $P_i \in \{ P_i \}$, let $Z_i$ be a nonempty, $P_i$-invariant, and $P_i$-cocompact subspace.

Let $\widetilde{X}$ denote the dual $CAT(0)$ cube complex of $( X , \mathcal{W} )$. For each $Z_i$, let $C(Z_i)$ be the convex subcomplex associated to $Z_i$. Then there exist a compact subcomplex $K \in C(X)$ such that

\begin{enumerate}

\item \label{HyperbolicCore} $C(X) = GK \cup (\mathbf{\cup}_i GC(Z_i))$,

\item \label{IsolatedPeripherals} $gC(Z_i) \cap C(Z_j) \subset GK$ unless $j=i$ and $g \in P_i$, and

\item \label{PeripheralCore} $P_i$ acts cocompactly on $C(Z_i) \cap GK$.

\end{enumerate}

\end{Theorem}

For a group action in the setting above, we say $G$ \emph{acts cocompactly on $\widetilde{X}$ relative to $ \{ C(Z_i) \}$}. 


The following Corollary is key to proving the implication (\ref{Mcocompact}) $\implies$ (\ref{Mchargeless}) of the main theorem.

\begin{Corollary} \label{InducedCubulations}

Let $(G , \{ P_i \} )$ be a relatively hyperbolic group. If $G$ is cocompactly cubulated then so each peripheral subgroup of $G$.

\end{Corollary}

\begin{proof}

Suppose $G$ acts properly and cocompactly on a $CAT(0)$ cube complex $\widetilde{X}$ and let $\mathcal{H}$ denote the collection of hyperplanes of $\widetilde{X}$. Cutting along any hyperplane decomposes $\widetilde{X}$ into two components (e.g. Theorem 2.13(4)) of \cite{wis12b}) thus $( \widetilde{X} , \mathcal{H} )$ is a wallspace. It is well-known that the dual cube complex of $( \widetilde{X} , \mathcal{H} )$ is isomorphic to $\widetilde{X}$ via an isomorphism that preserves the $G$-action, so we can use the same for both of them.

The action of $G$ on $\widetilde{X}$ preserves the wallspace structure. For each $H \in \mathcal{H}$, the subgroup $K = stab(H) \leq G$ acts cocompactly on $H$. Further, $H$ is isometrically embedded in $\widetilde{X}$ (e.g. Theorem 2.13(3) of \cite{wis12b}) implying $K$ is quasi-isometrically embedded in $G$. Thus $K$ is relatively quasiconvex in $G$ by Corollary 1.3 of \cite{hru10}. For each peripheral subgroup $P_i$, we can always find a $P_i$-invariant, $P_i$-cocompact subspace. E. g., choose any point $x_0$ and consider its $P_i$-orbit.

Thus our action satisfies the hypotheses of Theorem \ref{relcpt}. Since $G$ acts cocompactly, we choose the compact subcomplex $K$ so that $GK = \widetilde{X}$. Theorem \ref{relcpt}(\ref{PeripheralCore}) then implies $P_i$ acts cocompactly on the associated $C(Z_i) = C(Z_i) \cap GK$. Therefore $P_i$ is cocompactly cubulated.

\end{proof}

We can now prove the implication $(\ref{Mcocompact}) \implies (\ref{Mchargeless})$ in the main theorem.

\begin{Theorem}\label{CocompactImpliesChargeless} 
Suppose $M$ is a mixed manifold and $\pi_1M$ is virtually cocompactly cubulated. Then $M$ is chargeless. 
\end{Theorem}

\begin{proof}

By Theorem B of \cite{pw13}, it is sufficient to show that for any thick graph manifold cluster $N$ of $M$, its fundamental group $\pi_1N$ is virtually cocompactly cubulated. If $\widehat{M}$ is a finite-sheeted cover of $M$ with $\pi_1\widehat{M}$ cocompactly cubulated, then Corollary \ref{InducedCubulations} implies the fundamental group of each graph manifold cluster of $\widehat{M}$ is cocompactly cubulated. This implies the fundamental group of each graph manifold cluster of $M$ is virtually cocompactly cubulated. Therefore each graph manifold cluster of $M$ is chargeless by Theorem B of \cite{pw13} and hence so is $M$.

\end{proof}

\section{Surfaces in Chargeless Graph Manifold Clusters}\label{graph}

The goal of the next three sections is to prove Proposition \ref{step1} which states that a chargeless mixed manifold $M$ admits a frame efficient collection of surfaces. This section contains preliminary results used in Section \ref{cubulating} to construct a frame efficient collection. Recall from Section \ref{DefineEfficient} our strategy for constructing the properly immersed surfaces of a frame efficient collection is to attach surfaces from efficient collections in graph manifold clusters to surfaces in hyperbolic blocks true to the framing induced by the efficient collections together along boundary curves. In this section we define an efficient collection and prove Theorem \ref{graphsurfaces} which shows that the graph manifold clusters of $M$ admit efficient collections.

Hagen-Przytycki \cite{hp15} constructed in a chargeless graph manifold a collection of surfaces they call an efficient collection. The definition below highlights the key properties of the collection they constructed.

\begin{Definition} \label{graphefficient}

Let $N$ be a graph a manifold. A finite collection $\mathcal{S}$ of surfaces properly immersed imcompressible in $N$ which are in general position is an efficient collection if $\mathcal{S}$ has the following properties:

\begin{enumerate}

\item Each element of $\pi_1N$ has a cut-surface in $\mathcal{S}$.

\item All JSJ tori belong to $\mathcal{S}$.

\item Each $S \in \mathcal{S}$ is virtually embedded in $N$.

\item Let $B$ be a Seifert fibered block. The horizontal pieces in $B$ are all embedded and do not intersect one another.

\item Two horizontal pieces of a surface $S \in \mathcal{S}$ cannot be directly attached in the following sense: If $S_0 \subset S$ is a piece of $S$ embedded horizontally in a Seifert fibered block $B$ and $B'$ is a block adjacent to $B$ via a JSJ torus $T$, then each component of $S_0 \cap B'$ is a fiber of $B'$.

\item For a boundary torus $T$ in a block $B$, there are exactly two surfaces $S, S' \in \mathcal{S}$ intersecting $T$ where $S \cap B$ is horizontal and $S' \cap B$ is vertical.

\end{enumerate}

\end{Definition}

Theorem \ref{graphsurfaces} is minor modification of a result of Hagen-Przytycki \cite{hp15}. Hagen-Przytycki worked in a setting where Seifert fibered spaces are not considered graph manifolds. We extend their proof to cover a Seifert fibered space with boundary using a trivial version of their argument. Note that although sol manifolds can be treated as graph manifolds, they are excluded from the statement below since sol manifolds are not chargeless.

\begin{Theorem}[Hagen-Przytycki in \cite{hp15}] \label{graphsurfaces}

Let $N$ be either a chargeless graph manifold or a Seifert fibered space with boundary. Then $N$ admits an efficient collection of surfaces.

\end{Theorem}

The proof depends on the following Proposition.

\begin{Proposition} \label{horizontal}

Suppose $N$ is a chargeless graph manifold with at least one JSJ torus. Let $B$ be a Seifert fibered block of $N$ and $T_1, ... , T_k$ the JSJ tori of $N$ contained in $B$. For each $T_i$, let $B'_i$ denote the Seifert fibered block of $N$ adjacent to $B$ via $T_i$.
 There is a properly embedded horizontal surface $S \subset B$ with the following property: For each $T_i$, every component of $S \cap T_i$ is an $S^1$-fiber of $B'_i$.

\end{Proposition}

\begin{proof}

Obtain a Seifert fibered space $\bar{B}$ from $B$ using the following process: For every $T_i$, perform a Dehn filling along a fiber of $B'_i$. 

First assume $B$ is interior. In \cite{lw97}, Lueke and Wu define the Euler number of $B$ relative to the framing by the fibers of adjacent blocks to be the Euler number of $\bar{B}$. This relative Euler number differs from the charge by only a sign (See Section 1.3 of \cite{bs04} where they give a full definition of charge and observe this fact.) and is therefore $0$. Thus $\bar{B}$ has Euler number $0$. Proposition 2.2 of \cite{hat07} then implies $\bar{B}$ contains an embedded horizontal surface $S$. Each component of $S \cap T_i \subset \bar{B}$ bounds a disk in $\bar{B}$ is therefore isotopic to a fiber of $B'_i$ since we performed our Dehn filling along a fiber. It follows $S \cap B$ is the desired surface.

Assume now $B$ intersects $\partial{N}$. Then $\bar{B}$ has nonempty boundary, so Proposition 2.2 of \cite{hat07} implies it contains an embedded horizontal surface $S$ and we get the desired surface from the argument as above.




\end{proof}

We now prove Theorem \ref{graphsurfaces}.

\begin{proof}[Proof of Theorem \ref{graphsurfaces}]

The proof is by construction. Most of the proof is found in \cite{hp15}. We summarize the key steps, with careful attention to the additional case of a Seifert fibered space with boundary.

We consider thick and thin graph manifolds separately.

\textbf{Thick graph manifold.} We review Hagen-Przytycki's construction in \cite{hp15} with an added detail for a thick graph manifold which is a Seifert fibered space with boundary. Their collection of surfaces is built from smaller subcollections.

\underline{Turbine collection.} First assume $N$ contains at least one JSJ torus. For a Seifert fibered block $B$ of $N$, choose two copies of the embedded horizontal surface $S$ provided by Proposition \ref{horizontal}. Let $T$ be a JSJ torus intersecting $B$ and $B'$ the block adjacent to $B$ via $T$. Choose in $B'$ an embedded, vertical, non-$\partial$-parallel annulus $A$ with boundary contained in $T$. For each curve $C \subset S \cap T$, cap off the two copies of $C$ in the two copies of $S$ using a copy of $A$. Do this for every component of $S \cap T$ and every JSJ torus contained in $B$ to obtain a new surface $S'$. The \emph{turbine collection} consists of one surface of this type for every block of $N$.

If $N$ is a Seifert space space, choose any embedded horizontal surface. The turbine collection then consists of this one surface. 

\underline{Vertical collection.} For each block $B$, consider a finite cover $F\times S^1 \to B$ with $F$ a compact hyperbolic surface with boundary of positive genus. In $F$, choose a family of geodesic simple closed curves $\mathcal{C}$ with the following property: When $F$ is cut along every curve of $\mathcal{C}$, each component of the resulting space is either a closed disc or an annulus which contains a component of $\partial{F}$. We say $\mathcal{C}$ \emph{fills} $F$.

Consider the family of vertical tori whose base curves in $F$ corresponded to the curves in $\mathcal{C}$. If $B$ intersects $\partial{N}$ then for each component $T$ of $B \cap \partial{N}$ add a vertical non-$\partial$-parallel annulus whose boundary curves lie in $T$. Map these vertical surfaces down into $B$. Construct such a family of vertical surfaces in every block of $N$ to obtain the \emph{vertical collection}.

Let $\mathcal{S}$ be the collection consisting of all the surfaces in the turbine and vertical collections together with all the JSJ tori of $N$. As explained by Hagen-Przytycki \cite{hp15}, $\mathcal{S}$ contains a cut surface for every nontrivial element of $\pi_1N$ and its surfaces are all virtually embedded so $\mathcal{S}$ is an efficient collection.

\textbf{Thin graph manifold.} For a thin graph manifold $N = T \times I$, choose two embedded annuli which are not homotopic to each other. $N$ has many possible Seifert fibrations. Fixing a fibration we may assume one of the annuli is vertical and the other horizontal.

\end{proof}

\section{Virtually compact special cubulations of cusped hyperbolic 3-manifolds with restricted boundary slopes} \label{hyperbolic}

This section contains more preliminary results we use in Section \ref{cubulating} to construct a frame efficient collection of surfaces in a chargeless mixed manifold. Recall our strategy for constructing properly immersed surfaces in a frame efficient collection is to glue surfaces from efficient collections in the graph manifold clusters to surfaces true to the induced framing in the hyperbolic blocks together along boundary curves. In the previous section we constructed the surfaces used in the graph manifold clusters. The main goals of this section are to prove Theorem \ref{hypsurfaces} which establishes the existence of the surfaces used in the hyperbolic blocks and then to prove Theorem \ref{HyperbolicCube} by showing that the surfaces provided by Theorem \ref{hypsurfaces} produce a virtually compact special cubulation of a cusped hyperbolic 3-manifold group.

Przytycki-Wise \cite{pw13} showed, using Wise's virtually compact special cubulation for hyperbolic 3-manifolds \cite{w12a}, there exists a collection of properly immersed, geometrically finite surfaces in a hyperbolic 3-manifold $N$ which provides a cut-surface for every nontrivial element of $\pi_1N$. Our Theorem \ref{hypsurfaces} strengthens Theorem 4.1 of \cite{pw13} by allowing these surfaces to be chosen to be true to a given framing.

\begin{Theorem}\label{hypsurfaces}

Let $N$ be a framed hyperbolic 3-manifold whose boundary is a disjoint union of tori written $\partial{N} = T_1 \cup \cdots \cup T_k$. There is a finite collection $\mathcal{S}$ of surfaces properly immersed incompressible in $N$ which are in general position and geometrically finite such that $\mathcal{S}$ is true to the framing and contains a cut-surface for all the nontrivial elements of $\pi_1N$. Moreover, the surfaces of $\mathcal{S}$ have no accidental parabolics.

\end{Theorem}

Before proving Theorem \ref{hypsurfaces}, we need to prove Theorem \ref{recube} which is our more tightly constrained variation of Wise's virtually compact special cubulation in \cite{w12a}. 

\begin{Theorem}\label{recube}

Let $N$ be a hyperbolic 3-manifold whose boundary is nonempty, disjoin union of tori $\partial{N} = T_1 , ... , T_k$. In each $T_i$, choose nonhomotopic simple closed curves $C_i$ and $D_i$. Then $\pi_1N$ is virtually the fundamental group of a compact special cube complex $X$ with the property that if $H \subset X$ is an immersed hyperplane of $X$ and a conjugate of $\pi_1H \leq \pi_1X = \pi_1N$ intersects some $\pi_1T_i$, then that intersection lies in either $\pi_1C_i$ or $\pi_1D_i$.

\end{Theorem}

We are not yet ready to prove Theorem \ref{recube} but can outline the two-step strategy. Passing to a finite cover $\widehat{N}$, Wise found a properly embedded, incompressible, geometrically finite surface $S$ in $\widehat{N}$ which intersects each boundary torus. This adds one slope to each $T_i$, the slope of each component of $\partial{S} \cap T_i$. Our first step is Proposition \ref{stop} which states that we can choose $S$ so that it intersects each $T_i$ in curves with the same slope as the respective curve $C_i$. Cutting along this surface decomposes $\widehat{N}$ into a graph of spaces with corresponding graph of groups satisfying the hypotheses of Theorem 16.28 in \cite{w12a} where, in particular, each vertex group is word hyperbolic and virtually compact special. In the proof of Theorem 16.28 in \cite{w12a}, Wise constructs a virtually compact special cubulation of $\pi_1 \widehat{N}$. This process involves choosing a second curve in each $T_i$ which adds a second slope. Our second step is showing we can choose that second curve in each $T_i$ to be $D_i$ respectively.

The following is a modification of Proposition 4.6 of \cite{pw13}.

\begin{Proposition}\label{stop}
There is a finite index cover $\widehat{N} \to N$ with a properly embedded incompressible, possibly disconnected surface $S \subset \widehat{N}$ where each component is geometrically finite, and $S$ has the following property: For each boundary torus $T_i$ of $N$ and each boundary torus $\widehat{T}_{i_j}$ of $\widehat{N}$ covering $T_i$, the surface $S$ has a nonempty intersection with $\widehat{T}_{i_j}$ consisting of parallel copies of a curve in $\widehat{T}_{i_j}$ covering $C_i$. Further, the components of $S$ contain no accidental parabolics.

\end{Proposition}

\begin{proof}

In the proof of Proposition 4.6 in \cite{pw13}, Przytycki-Wise find a finite cover $\widehat{N}$ and a properly embedded, possibly disconnected surface $S' \subset \widehat{N}$ with the following properties: Each component of $S'$ is incompressible and geometrically finite. Further, for each $T_i$ and each JSJ torus $T_{i_j}$ of $\widehat{N}$ covering $T_i$, the intersection $S' \cap T_{i_j}$ is nonempty and consists of parallel copies of a curve covering $C_i$.

The components of $S'$ might contain accidental parabolics, but by Lemma 14.22 and Remark 14.23 of \cite{w12a} there is a properly embedded surface $S$ with each component geometrically finite such that $S$ has no accidental parabolics and $\partial{S}' \subset \partial{S}$. If an embedded surface intersects a boundary torus in multiple curves, they have to have the same slope, so the components of $\partial{S} - \partial{S}'$ do not add new slopes to the boundary tori.


\end{proof}

We now show how to modify the proof of Theorem 16.28 in \cite{w12a} to prove our Theorem \ref{recube}.

\begin{proof}[Proof of Theorem \ref{recube}]

Let $\widehat{N}$ and $S$ be the finite cover and embedded surface guaranteed by Proposition \ref{stop} where we choose $S$ so that for each boundary torus $T_{i_j}$ of $\widehat{N}$ covering a boundary torus $T_i$ of $N$, the intersection $S \cap T_{i_j}$ is nonempty and consists of closed curves with the same slope as a curve covering $C_i$. To simplify notation assume $\widehat{N} = N$. Passing to a further finite cover we may assume that for each boundary torus $T_i$ that $C_i$ and $D_i$ form a basis for $\pi_1T_i$.

As in the proof of Theorem 14.29 of \cite{w12a}, $S$ splits $\pi_1N$ as a graph of groups where each vertex group is word hyperbolic and virtually compact special and each edge group is quasi-isometrically embedded. The proof of Theorem 16.28 in \cite{w12a} extends this splitting to a cubulation of $\pi_1N$. The proof involves a lot of machinery, so we just explain how to modify this proof in order to make an arbitrary choice for the second slope in each boundary torus.

Choose a component $E$ of $S$ and let $T_i$ be a boundary torus intersecting $E$. In step 3 of Wise's proof, he gives the torus $T_i$ a cubical structure with $C_i$ as a $1$-cube and attaches it to a cubulation of a subgroup of $\pi_1N$ called the \emph{expanded edge group $\pi_1E^+$ of $\pi_1E$}. Obtain the extended edge space $E^{+}$ of $E$ by taking the union of $E$ together with all the boundary tori intersecting $E$. The subgroup $\pi_1E^+ \leq \pi_1N$ is the \emph{expanded edge group} of $\pi_1E$. It consists of $\pi_1E$ together with multiple HNN extensions corresponding to the intersections of $\pi_1E$ with the tori subgroups of $\pi_1N$.

Equip $T_i$ with a cubical structure that uses $C_i$ and $D_i$ as $1$-cubes. Choose a compact cubulation $B$ of $E$.  The group $\pi_1E$ is free so it's well-known that such cubulations exist. Attach $T_i$ to $B$ along a local isometry representing  $\pi_1C_i \to \pi_1E$. This may require subdividing the cubical structure for $T_i$. (In general, Wise passes to a further finite index subgroup $\pi_1N$ before carrying out this process.)

Doing this for every boundary torus interesting $E$ yields cocompact cubulation $\pi_1E^+$. Repeat this process for every component of $S$ noting that each boundary torus intersects some component of $S$. The rest of the proof in \cite{w12a} extends these cubulations to a virtually compact special cubulation of $\pi_1N$.

\end{proof}

Having proved Theorem \ref{recube}, we can now prove Theorem \ref{hypsurfaces}.

\begin{proof}[Proof of Theorem \ref{hypsurfaces}]

Most of the proof is found as the proof of Theorem 4.1 of \cite{pw13}. We outline the construction and explain how the slopes in the tori are controlled.

In each boundary torus $T_i$ choose simple closed curves $C_i$ and $D_i$. We may assume, by passing to a finite cover and applying Theorem \ref{recube}, that $\pi_1N = \pi_1X$ for a compact special cube complex $X$ with universal cover $\widetilde{X}$ and that any conjugate of an immersed hyperplane subgroup of $\pi_1X$ only intersecting some $\pi_1T_i$ intersects in either $\pi_1C_i$ or $\pi_1D_i$.

Let $g \in \pi_1N$. Choose an axis for $g$ in $\widetilde{X}$ and a hyperplane $\widetilde{H}$ which intersects that axis transversally. Let $K = stab(\widetilde{H}) \leq \pi_1N$.

Let $\widetilde{N}$ be the universal cover of $N$ and let $\widehat{N} \to N$ be the $\pi_1\widetilde{H}$ cover of $N$. Let $L$ be the hyperbolic convex core of the cover $\widehat{N}$ and $\widetilde{L}$ an elevation of $L$ to $\widetilde{N}$. Since $\widetilde{N}$ and $\widetilde{X}$ are quasi-isometric, any axis for $g$ intersects $\widetilde{L}$ transversally. Thus there is an elevation $\widetilde{S} \subset \partial{\widetilde{N}}$ of a component $S$ of $\partial{N}$ which intersects the axis transversally. Immersing $S$ into $N$ gives us a cut-surface for $g$. As explained in the proof of Theorem 4.1 of \cite{pw13}, $S$ is geometrically finite and can be made to have no accidental parabolics.

\end{proof}

We now prove Theorem \ref{HyperbolicCube} by showing that the collection of surfaces provided by Theorem \ref{hypsurfaces} is dual to a virtually compact special cubulation.

\begin{proof}

Let $\mathcal{S}$ be the collection of surfaces provided by Theorem \ref{hypsurfaces} that is true to the framing on $N$. The action of $\pi_1M$ on $\widetilde{X}$ is free and proper since $\mathcal{S}$ contains a cut-surface for every non-trivial element of $\pi_1N$.  Each hyperplane subgroup of $\pi_1X$ is commensurable to a conjugate of a surface subgroup $\pi_1S$ with $S \in \mathcal{S}$ so the interactions with tori subgroups are as desired. Since the surfaces intersect each boundary torus in only two slopes of curves, it is a straightforward exercise applying Theorem \ref{relcpt} to show the action is cocompact.

It remains to show $X$ is special. Passing to finite-sheeted cover, we can assume $\pi_1X = \pi_1M$ is special. Since $X$ is compact, the hyperplane subgroups are quasi-isometrically embedded, so by Theorem 16.23 of \cite{w12a} they are separable in $\pi_1X$ and satisfy double coset separability. Theorem 9.19 of \cite{hw08} then implies $X$ is virtually special.

\end{proof}

In Section \ref{DefineEfficient} we mentioned that Przytycki-Wise used additional ``capping off'' surfaces to construct a collection of surfaces satisfying Definition \ref{efficient} (1)-(5). When constructing a frame efficient collection, we use their capping off surfaces in the hyperbolic blocks.

\begin{Proposition}[Prop. 4.6 of \cite{pw13}]\label{cap}

Let $C_1$, ... , $C_k$ be essential closed curves in the respective boundary tori $T_1$, ... , $T_k$ of $N$. There exists a geometrically finite immersed incompressible surface $S \to N$ with $S \cap \partial{N}$ covering $C_1$ such that all the parabolic elements of $\pi_1S$ are conjugate into some $\pi_1C_i$.

\end{Proposition}


\section{Constructing a Frame Efficient Collection} \label{cubulating}


In this section we prove Proposition \ref{step1}, which states that a chargeless mixed manifold $M$ admits a frame efficient collection of surfaces. In the previous two sections, we constructed an efficient collection of surfaces in the graph manifold clusters of $M$ and a collection of surfaces in each hyperbolic block true to the framing induced by the efficient collections. We construct properly immersed surfaces in a frame efficient collection by attaching surfaces from these collections together along their boundary curves. To illustrate the two key challenges of this process, suppose $S$ and $S'$ are both surfaces immersed in $M$ with some of their boundary components mapping into a transitional torus $T$. Suppose at least some of those boundary components have the same slope. The first challenge is that these boundary curves of $S$ and $S'$ might map onto their images with different degrees. The second is that $S$ and $S'$ might have different numbers of boundary components mapping into $T$ with that particular slope.

We deal with matching the degrees in Lemma \ref{degree}, which is a fact Przytycki-Wise \cite{pw13} proved as part of the proof of their Theorem 2.1. A main part of the proof of Proposition \ref{step1} involves matching the multiplicities.

\begin{Lemma}[Matching the degrees.] \label{degree}

Let $M$ be a 3-manifold with boundary and $\mathcal{S}$ a collection of surfaces immersed in $M$, but not necessarily properly immersed.  There is a constant $d > 0$ with the following property: For every $S \in \mathcal{S}$ and every boundary arc $C \subset \partial{S}$ there is a finite cover $\widehat{S} \to S$ so that for any curve $\widehat{C} \subset \partial{\widehat{S}}$ covering $C$, the map $\widehat{C} \to M$ obtained by restricting the composition $\widehat{S} \to S \to M$ maps onto it image with degree $d$.

\end{Lemma}



The second challenge to constructing properly immersed surfaces in a frame efficient is the main focus of the proof of Proposition \ref{step1}. Before proving this result, we prove the following which gathers all the surfaces we use in the graph manifold clusters and hyperbolic blocks to construct a frame efficient collection.


\begin{Lemma}[Gathering materials in preparation for construction.] \label{efficientpieces}

Let $M$ be a chargeless mixed manifold. There is a collection of surfaces $\mathcal{S}$ immersed in $M$ with the following properties:

\begin{enumerate}

\item Each $S \in \mathcal{S}$ is either properly immersed in a graph manifold cluster or properly immersed in a hyperbolic block. Note they are not, in general, properly immersed in $M$.

\item For each hyperbolic block or graph manifold cluster $N$ of $M$ and each nontrivial element $g \in \pi_1N$, there is a surface $S \in \mathcal{S}$ immersed in $N$ which is a cut-surface for $g$.

\item All JSJ tori belong to $\mathcal{S}$.

\item Each $S \in \mathcal{S}$ immersed in a graph manifold cluster is virtually embedded.

\item Each $S \in \mathcal{S}$ in a hyperbolic block is geometrically finite.

\item Two horizontal pieces cannot be directly attached in the following sense: Let $S \in \mathcal{S}$ and $B$ be a Seifert fibered block. Suppose $S$ has a piece $S_0$ immersed in $B$ which is horizontal and $B'$ is a Seifert block adjacent to $B$ via a JSJ torus $T$. Then each component of $S_0 \cap T$ is an $S^1$-fiber of $B'$.


\item Let $B$ be a Seifert fibered block. The images of horizontal pieces immersed in $B$ are disjoint. Further, each horizontal piece is the composition of a covering map between surfaces and an embedding into $M$.


\item The accidental parabolics are all vertical in the following sense: Suppose $S \in \mathcal{S}$ contains an accidental parabolic $C \subset S$ freely homotopic into a transitional torus $T$. Let $B$ denote the Seifert fibered block adjacent to $T$. Then the image of $C$ is freely homotopic to a fiber of $B$.

\item $\mathcal{S}$ includes ``capping off'' surfaces. Suppose $S \in \mathcal{S}$ and a boundary curve $C \subset \partial{S}$ maps into a transitional torus $T$. Then there is another surface $S' \in \mathcal{S}$ immersed in the hyperbolic block adjacent to $T$ so that $\partial{S}$ maps only into $T$ and the image consists of closed curves with the same slope as $C$.

\item There is a uniform degree $d$ so that for every surface $S \in \mathcal{S}$, the immersion $S \to M$ maps each boundary curve $C \subset \partial{S}$ onto its image with degree $d$. Further, if $C$ maps into a transitional torus $T$, then there is exactly one surface $S'$ in the graph manifold cluster adjacent to $T$ such that $S' \cap T$ is nonempty consisting of curves homotopic to $C$.

\end{enumerate}
\end{Lemma}

\begin{proof}

Each graph manifold cluster of $M$ is either a Seifert fibered space with boundary or a chargeless graph manifold and therefore admits a collection of surface satisfying Theorem \ref{graphsurfaces} and hence criteria (1), (3), (5), and (6) of our lemma. Let $\mathcal{S}_1$ denote the union of these collections over all the graph manifold clusters of $M$. For each transitional torus $T$ of $M$, there are exactly two surfaces $S, S' \in \mathcal{S}_1$ which intersect $T$. Let $\alpha_T$ and $\beta_T$ denote the slope of the components of $S \cap T$ and $S' \cap T$ respectively.

The surfaces in $\mathcal{S}_1$ put a framing on every hyperbolic of $M$. By Theorem \ref{hypsurfaces}, we can choose in every hyperbolic block a collection of properly immersed satisfying criteria (1) and (4) of our lemma which are true to the framing induced by surfaces in $\mathcal{S}_1$. Further, these surfaces have no accidental parabolics. Let $\mathcal{S}_2$ denote the union of these collections over all the hyperbolic blocks of $M$.


For each transitional torus $T$ of $M$, Proposition \ref{cap} says the hyperbolic block adjacent to $T$ contains a properly immersed, geometrically finite surface $S$ whose boundary curves all map into $T$ and have slope $\alpha_T$. Further, any accidental parabolic of $S$ is freely homotopic to an $S^1$-fiber of an adjacent Seifert fibered block. We can construct a capping off surface for the slope $\beta_T$ similarly.

Let $\mathcal{S}$ consists of all the surfaces in $\mathcal{S}_1$ and $\mathcal{S}_2$ together with the capping off surfaces constructed above. By Lemma \ref{degree} we may replace each surface of $\mathcal{S}$ with a finite cover as needed so that their boundary curves map onto their images with a uniform degree. Note that after applying this process $\mathcal{S}$ still satisfies criteria (7).



\end{proof}

We are now ready to prove Proposition \ref{step1}, which involves matching the multiplicities.

\begin{proof}[Proof of Proposition \ref{step1}]

This proof is similar to the proof of Theorem 2.1 in \cite{pw13} but in order to ensure cocompactness we need to be more delicate in certain places. Let $\mathcal{S}$ be the collection of surfaces guaranteed by Lemma \ref{efficientpieces}. For each transitional torus $T$ of $M$, the boundary curves of the surfaces of $\mathcal{S}$ which map into $T$ form two families of curves with the same slope. Label these slopes $\alpha_T$ and $\beta_T$ respectively.

Let $\mathcal{H}$ and $\mathcal{G}$ denote the collections of all the hyperbolic blocks of $M$ and all the graph manifold clusters of $M$ respectively. For each hyperbolic block $Q \in \mathcal{H}$, let $\mathcal{S}_Q \subset \mathcal{S}$ denote the subcollection of surfaces which map into $Q$. For each $N \in \mathcal{G}$, define $\mathcal{S}_N$ similarly.

Choose $Q \in \mathcal{H}$ and $S \in \mathcal{S}_Q$. If $T$ is a transitional torus with $S \cap T$ nonempty, let $N$ be the graph manifold cluster containing $T$ and $S', S'' \in \mathcal{S}_N$ the pair of surfaces intersecting $T$ in curves of slope $\alpha_T$ and $\beta_T$ respectively. If $S \cap T$ contains components with slope $\alpha_T$ then let $m$ denote the number of such components and $n$ denote the number of components of $S' \cap T$. Choose $k$ so that $k n \leq m$ and add $k$ copies of $S'$ to $N$ attaching them to $S$ so that every boundary component of $S$ is attached to a copy of $S'$. If $S \cap T$ contains components with slope $\beta_T$ follow the same process using $S''$. Repeat this for every transitional torus intersecting $\partial{S}$ to obtain a new surface $S^*$.

We now attach ``capping off'' surfaces to $S^*$. Let $\mathcal{T}$ denote the collection of transitional tori which intersect $\partial{S^*}$. For each $T \in \mathcal{T}$, let $S^T_{\alpha}$ and $S^T_{\beta}$ denote the capping off surfaces in $\mathcal{S}$ which intersect $T$ in curves with slope $\alpha_T$ and $\beta_T$ respectively. Further, let $r_T$ and $s_T$ denote the the number of components of $S^* \cap T$ with slope $\alpha_T$ and $\beta_T$ respectively. Let $a_T$ and $b_T$ denote the number of components of $\partial{S^T_{\alpha}}$ and $\partial{S^T_{\beta}}$ respectively. Note some $r_T$ and $s_T$ could be $0$.




If every $r_T$ and $s_T$ were nonzero we could define $\ell$ to be the least common multiple of all the constants given above; ie, define $\ell = \lcm( \{ r_T, s_T, a_T, b_T \}_{T \in \mathcal{T}})$. (Yikes! That's a lot of numbers.) In general,  we define $\ell$ to be the lcm of the ones that are nonzero.

Take $\ell$ copies of $S^*$, and, for each $T \in \mathcal{T}$, take $ r_T \ell / a_T$ of copies of $S^T_{\alpha}$ and $s_T \ell / b_T$ copies of $S^T_{\beta}$. Attaching these surfaces together creates a surface properly immersed in $M$. Repeat this process for every $Q \in \mathcal{H}$ and every $S \in \mathcal{S}_Q$ to obtain a collection $\mathbf{S}'$ of surfaces properly immersed in $M$. If $N \in \mathcal{G}$ and $S \in \mathcal{S}_N$ intersects a transitional torus, then some surface of $\mathbf{S}'$ has an piece which is a copy of $S$, so we do not repeat this gluing process for surfaces in the graph manifold clusters. Construct a frame efficient collection $\mathbf{S}$ from $\mathbf{S}'$ by adding the surfaces from $\mathcal{S}$ which were already properly immersed in $M$ as well as all the transitional tori of $M$. (The other JSJ tori were already in $\mathcal{S}$.)

The surfaces of $\mathbf{S}$ are properly immersed incompressible, in general position, and satisfy Definition \ref{efficient}(\ref{vembedded})-(\ref{geometricallyfinite}) and (\ref{turbine})-(\ref{verticalaccidental}) since the surfaces from Lemma \ref{efficientpieces} in have these properties. The collection $\mathbf{S}$ contains a cut-surface for every nontrivial element of $\pi_1M$ since it contains all the JSJ tori and the pieces of surfaces in $\mathbf{S}$ provide cut-surfaces in all the blocks of $M$. The proof in \cite{pw13} that their collection of surfaces satisfies the strong separation goes through without change for our collection $\mathbf{S}$ so we refer the reader there. From all this it follows $\mathbf{S}$ is a frame efficient collection.

\end{proof}


\section{Dual Cube Complex of a Frame Efficient Collection} \label{mainproof}

In Section \ref{cubulating} we constructed a frame efficient collection of surfaces in a chargeless mixed manifold $M$. In this section we prove Proposition \ref{step2} which states that for a mixed manifold $M$ admitting a frame efficient collection, $\pi_1M$ is virtually compact special. Most of the work in the proof is verifying cocompactness. We will use Theorem \ref{relcpt} to show the action of $\pi_1M$ on $\widetilde{X}$ is cocompact relative to a collection of convex subcomplexes associated to the graph manifold subgroups described in Section \ref{review}. We will then show that each graph manifold subgroup acts cocompactly on its associated convex subcomplex.


When studying the convex subcomplex associated to a graph manifold subgroup $\pi_1N$, the proof is simplest when every surface containing an accidental parabolic homotopic into $N$ actually intersects $N$. The following will allow us to assume we are always in this case:

\begin{Proposition} \label{push}

Let $M$ be a mixed manifold admitting a frame efficient collection of surfaces $\mathbf{S}$. Then $M$ also admits an efficient collection $\mathbf{S}'$ with the following property: If $S \in \mathbf{S}$ contains an accidental parabolic $C \subset S$ then $C$ is homotopic in $S$ to a curve that maps into a thin graph manifold block $N$. In particular, the image of $S$ intersects $N$.

\end{Proposition}





\begin{proof}

Let $S \in \mathbf{S}$ and supposed $C \subset S$ is an accidental parabolic. Apply the accidental parabolic removal process described in the proof of Lemma 14.32 of \cite{w12a} to obtain a new surface $S'$ with two new boundary components intersecting $T$. Cap off these boundary components with a vertical annulus in $B$. If $B$ is a Seifert-fibered block of a thick graph manifold, then we can choose the annulus so that it is not $\partial$-parallel and therefore eliminate the accidental parabolic.

\end{proof}

To prove Proposition \ref{step2}, we also need the following Lemma:

\begin{Lemma} \label{close}

Let $M$ be a mixed manifold with universal cover $\widetilde{M}$. Let $\mathbf{S}$ be a frame efficient collection and $\widetilde{\mathbf{S}}$ all the elevations of surfaces in $\mathbf{S}$ to $\widetilde{M}$.  There exists $R > 0$ so that the following holds: Let $\widetilde{S}, \widetilde{S}' \in \widetilde{\mathbf{S}}$ and let $\widetilde{B}$ be a Seifert fibered block of $\widetilde{M}$. If both $\widetilde{S} \cap \widetilde{B}$ and $\widetilde{S}' \cap \widetilde{B}$ are non-empty and are distance at least $R$ from each other, then $\widetilde{S}$ and $\widetilde{S}'$ do not intersect in $\widetilde{M}$.

\end{Lemma}

The proof of Lemma \ref{close} uses the following two lemmas due to Przytycki-Wise \cite{pw13}:

\begin{Lemma}[Lemma 2.5 of \cite{pw13}]\label{r1}
Let $\mathcal{S}$ be a finite family of geometrically finite immersed incompressible surfaces in a compact hyperbolic 3-manifold $N$. Let $\widetilde{N}$ denote the universal cover of $N$ and $\widetilde{\mathcal{S}}$ all the elevations of surfaces in $\mathcal{S}$ to $\widetilde{N}$. There exists $R' = R'(N, \mathcal{S})$ such that if the stabilizer of an elevation $\widetilde{S} \in \widetilde{\mathcal{S}}$ intersects the stabilizer of a boundary plane $\widetilde{T} \subset \partial{N}$ along an infinite cyclic group, then $N = N_R(\widetilde{S}) \cap \widetilde{T}$ is nonempty.

Moreover, assume that we have two such elevations $\widetilde{S}$, $\widetilde{S}'$ of possibly distinct surfaces. If $\widetilde{S} \cap \widetilde{T}$ and $\widetilde{S}' \cap \widetilde{T}$ are nonempty and at distance $\geq R$ in the intrinsic metric on $\widetilde{T}$  
(resp. $\mathcal{N}_R(\widetilde{S}) \cap \widetilde{T}$ and $\mathcal{N}_R(\widetilde{S}') \cap \widetilde{T}$ are sufficiently far with respect to some $r$), then $\widetilde{S}$ and $\widetilde{S}'$ are disjoint. (resp. at distance $\geq r$) and $\widetilde{T}$ is the only boundary plane of $\widetilde{N}$ intersecting both $\widetilde{S}$ and $\widetilde{S}'$.

\end{Lemma}

\begin{Lemma}[Remark 3.6 of \cite{pw13}]\label{r2}

Let $\mathcal{S}$ be a finite family of immersed incompressible surfaces in a thick graph manifold $N$. There exists $R' = R'(N, \mathcal{S})$ with the following property: Let $B \subset N$ be a Seifert fibered block with elevation $\widetilde{B} \subset \widetilde{N}$ and let $\widetilde{S}$, $\widetilde{S}'$ be elevations to $\widetilde{N}$ of surfaces in $\mathcal{S}$. Suppose $\widetilde{S}_o = \widetilde{S} \cap \widetilde{B}$ and $\widetilde{S}'_o = \widetilde{S}' \cap \widetilde{B}$ are both vertical and that there is JSJ-plane $\widetilde{T} \subset \partial{\widetilde{B}}$ intersecting both $\widetilde{S}_o$ and $\widetilde{S}'_o$. If the distance between the lines $\widetilde{S}_o \cap \widetilde{T}$ and $\widetilde{S}'_o \cap \widetilde{T}$ is $\geq R$ in the intrinsic metric on $\widetilde{T}$, then $\widetilde{S}_o$ and $\widetilde{S}'_o$ are disjoint and $\widetilde{T}$ is the only JSJ-plane contained in $\widetilde{B}$ intersecting both $\widetilde{S}$ and $\widetilde{S}'$.


\end{Lemma}


\begin{proof} [Proof of Lemma \ref{close}]

For each graph manifold cluster or hyperbolic block $N$ of $M$, let $\mathcal{S}_N$ denote the collection of all pieces of surfaces in $\mathbf{S}$ which map into $N$. Choose $R'$ satisfying Lemma \ref{r1} and Lemma \ref{r2} for every pair $(N, \mathcal{S}_N)$.

Let $\widetilde{B}$ be a Seifert fibered block of $\widetilde{M}$. Suppose $\widetilde{S}, \widetilde{S}' \in \widetilde{\mathbf{S}}$ both intersect $\widetilde{B}$. If there is no JSJ-plane contained in $\widetilde{B}$ which insersects both $\widetilde{S}$ and $\widetilde{S}'$ then $\widetilde{B}$ is only block of $\widetilde{M}$ intersecting both $\widetilde{S}$ and $\widetilde{S}'$. In this case, being far apart in $\widetilde{B}$ clearly implies $\widetilde{S}$ and $\widetilde{S}'$ do not intersect in $\widetilde{M}$.

Now assume a JSJ-plane $\widetilde{T}$ of $\widetilde{B}$ intersects both $\widetilde{S}$ and $\widetilde{S}'$. Let $\widetilde{B}'$ be the block adjacent to $\widetilde{B}$ via $\widetilde{T}$. We will find $R > 0$ so that if $\widetilde{S} \cap \widetilde{T}$ and $\widetilde{S}' \cap \widetilde{T}$ are distance at least $R$ from each other then $\widetilde{S}$ and $\widetilde{S}'$ do not intersect in $\widetilde{B}'$ and that $\widetilde{T}$ is the only JSJ-plane of $\widetilde{B}$ intersecting both $\widetilde{S}$ and $\widetilde{S}'$. It will follow that that $\widetilde{S}$ and $\widetilde{S}'$ do not intersect in the component of $\widetilde{M} - \widetilde{B}$ containing $\widetilde{B}'$. Since $T$ was chosen arbitrarily, it will then follow that $\widetilde{S}$ and $\widetilde{S}'$ do not intersect in $\widetilde{M}$.



There are a few cases to consider. First, if one of these elevations has a horizontal piece in $\widetilde{B}$ and the other a vertical piece in $\widetilde{B}$, then they cross in $\widetilde{B}$. Therefore $d ( \widetilde{S} \cap \widetilde{B} , \widetilde{S}' \cap \widetilde{B} ) = 0 < R'$.

Next assume both $\widetilde{S} \cap \widetilde{B}$ and $\widetilde{S}' \cap \widetilde{B}$ are horizontal and that $\widetilde{B}'$ is a hyperbolic block. If $\widetilde{S} \cap \widetilde{T}$ and $\widetilde{S} \cap \widetilde{T}$ are both non-empty and at distance at least $R'$ from each other, then by Lemma \ref{r1} the elevations $\widetilde{S}$ and $\widetilde{S}'$ do not intersect in $\widetilde{B}'$ nor do they intersect any common JSJ planes of $\widetilde{B}'$ other than $\widetilde{T}$. Thus $\widetilde{S}$ and $\widetilde
 {S}'$ do not intersect in this component of $\widetilde{M} - \widetilde{B}$. If $\widetilde{B}'$ is a Seifert fibered block, then $\widetilde{S} \cap \widetilde{B}'$ and $\widetilde{S}' \cap \widetilde{B}'$ are vertical in $\widetilde{B}'$ and the same conclusion holds by applying Lemma \ref{r2}. 

Now assume they are both vertical in $\widetilde{B}$. If $\widetilde{B}'$ is hyperbolic, $R'$ still suffices by the same argument as before. The trickier case is when $\widetilde{B}'$ is Seifert fibered since then $\widetilde{S}$ and $\widetilde{S}'$ have pieces horizontal in $\widetilde{B}'$. The horizontal pieces in $\widetilde{B}'$ do not intersect and it follows that only finite many can lie $R'$-close to $\widetilde{S}$. Thus, there is some $R \geq R'$ so that if $\widetilde{S} \cap \widetilde{T}$ and $\widetilde{S}' \cap \widetilde{T}'$ are at distance at least $R$, then $\widetilde{S} \cap \widetilde{B}'$ and $\widetilde{S}' \cap \widetilde{B}'$ are at distance at least $R'$ and a previous case implies $\widetilde{S}$ and $\widetilde{S}'$ do not intersect in $\widetilde{M}$. This choice of $R$ depended only on $R'$ and the $\pi_1M$-orbit of $\widetilde{T}$. Therefore, we can choose $R$ uniformly.


\end{proof}

We now prove Proposition \ref{step2}, completing the proof of the main theorem.

\begin{proof}[Proof of Proposition \ref{step2}]

Let $M$ be a mixed manifold admitting a frame efficient collection of surfaces $\mathbf{S}$. Let $\widetilde{\mathbf{S}}$ the collection of all elevations of surfaces in $\mathbf{S}$ to $\widetilde{M}$, the universal cover of $M$. Choose a Riemannian metric for $M$ and lift it to $\widetilde{M}$. Assume $\mathbf{S}$ is as in Proposition \ref{push}. Let $\widetilde{X}$ be the $CAT(0)$ cube complex dual to the wallspace $( \widetilde{M} , \widetilde{\mathbf{S}} )$.

Let $N_1, \ldots, N_k$ denote the graph manifold clusters of $M$. For each graph manifold cluster $N_i$, choose an elevation $\widetilde{N}_i$ in $\widetilde{M}$ and let $P_i = stab(\widetilde{N}_i)$. 

Przytycki-Wise have shown that Definition \ref{efficient}(1)-(5) imply that $\pi_1M$ acts freely and properly on $\widetilde{X}$, that $\widetilde{X} / \pi_1M$ is virtually special, and that $\pi_1M$ acts cocompactly relative to $\{ C(\widetilde{N}_i) \}$. (See Theorems 2.1 and 2.4 of \cite{pw13}.) Therefore, it remains only to show that each $P_i$ acts cocompactly on its respective $C(\widetilde{N}_i)$.

Let $R >0$ be a constant satisfying Lemma \ref{close} for $M$ and $\mathbf{S}$. Fix a choice of $i$. To show $P_i$ acts cocompactly $C(\widetilde{N}_i)$, we find $L > 0$ so that each collection of pairwise crossing walls in $\mathbf{\widetilde{S}}_i$ has an $L$-center in $\widetilde{N}_i$. Since $P_i$ acts cocompactly on $\widetilde{N}_i$, Lemma \ref{center} will then imply there are finitely many $P_i$-orbits of pairwise crossing walls. First we show each wall in $\mathbf{\widetilde{S}}_i$ intersects $\widetilde{N}_i$. If $\widetilde{S} \in \mathbf{\widetilde{S}}_i$, then $stab(\widetilde{S}) \cap P_i$ is infinite. If $S$ is the surface covered by $\widetilde{S}$, this implies $S$ either intersects $N_i$ or has an accidental parabolic homotopic into $N_i$. Since $\mathbf{S}$ satisfies Proposition \ref{push}, $S$ intersects $N_i$ in the latter case.


Let $\mathcal{V} \subset \mathbf{\widetilde{S}}_i$ be a collection of walls which pairwise cross in $\widetilde{M}$. Now we show there is a Seifert fibered block of $\widetilde{N}_i$ intersecting every surface of $\mathcal{V}$. The JSJ-planes of $\widetilde{M}$ give it the structure of a tree of spaces. Let $\Gamma$ denote the tree dual to the collection of JSJ-planes of $\widetilde{M}$ in which there is a vertex for each block of $\widetilde{M}$ and an edge whenever two blocks are adjacent via a JSJ-plane. For each wall in $W \in\mathcal{V}$ consider the subtree of $\Gamma_W$ whose vertices corresponded to blocks intersecting that wall and edges to JSJ-planes intersecting that wall. These subtrees associated to walls in $\mathcal{V}$ pairwise intersect since the surfaces pairwise intersect in $\widetilde{M}$. Further, each of these subtrees intersects the subtree $\Gamma_{\widetilde{N}_i}$ consisting of all the blocks and JSJ-planes contained in $\widetilde{N}_i$. The Helly Property for trees states that if a collection of a subtrees in a tree pairwise intersect, then the total intersection of the collection of subtrees is nonempty. Therefore, $(\cap_{W \in \mathcal{V}} \Gamma_W) \cap \Gamma_{\widetilde{N}_i}$ is nonempty which implies there is a Seifert fibered block $\widetilde{B}$ of $\widetilde{N}_i$ intersecting every wall of $\mathcal{V}$.


Let $\mathcal{V}'$ denote the collection of all pieces in $\widetilde{B}$ of walls in $\mathcal{V}$. By our choice of $R > 0$, the pieces in $\mathcal{V}'$ are pairwise $R$-close.

We find an $\mathbb{R}$-fiber of $\widetilde{B}$ close to the vertical pieces of $\mathcal{V}'$ then a point on that fiber close to the horizontal pieces. Since $B$ is a Seifert fibered space with boundary, $\widetilde{B}$ has a product structure of the form $E \times \mathbb{R}$ respecting the $\mathbb{R}$-fibering of $\widetilde{B}$. If $N_i$ is thick, $E$ is a convex subset of $\mathbb{H}^2$. If $N_i$ is thin, $\widetilde{B} = \widetilde{N}_i$ and $E$ is an infinite strip. In either case, $E$ is $\delta$-hyperbolic. Projecting the vertical pieces of $\mathcal{V}'$ onto $E$ yields a collection of pairwise-close quasigeodesics in $E$. By Lemma \ref{HyperbolicCenter}, there is a constant $L_1 = L_{1}(B , \mathbf{S})$ and an $L_1$-center $y_0 \in E$ for these quasigeodesics.

Let $\ell = \{y_0\} \times \mathbb{R}$ be the $\mathbb{R}$-fiber $L_1$-close to the vertical pieces of $\mathcal{V}'$. The horizontal pieces are all disjoint so there is an upper bound $K = K(B, \mathbf{S} )$ on the size of any collection of pairwise $R$-close horizontal pieces. Choose a point $x_0 \in \ell$ lying on any horizontal piece in $\mathcal{V}'$. Let $f$ be the length of a regular fiber of $B$ and let $L_2 = fK$. Any horizontal piece in $\mathcal{V}'$ in $B$ is $L_2$-close to $x_0$.

Thus, $x_0$ is an $L_2$-center for $\mathcal{V}$. The constants we chose depended only on $B$, the Seifert fibered block covered by $\widetilde{B}$, and $\mathbf{S}$, so we can choose them uniformly. Therefore, $P_i$ acts cocompactly on $C(\widetilde{N}_i)$ by Lemma \ref{center}. This together with Theorem \ref{relcpt} implies $\pi_1M$ acts cocompactly on $\widetilde{X}$.
Therefore $X = \widetilde{X} / \pi_1M$ is virtually compact special.

\end{proof}

\section{Classification of virtually compact special 3-manifold groups} \label{finale}

We conclude by proving Theorem \ref{classification}, which gives a classification of virtually compact special 3-manifold groups in terms of their geometric structure.

\begin{proof}[Proof of Theorem \ref{classification}]

The case where $M$ is nongeometric follows from the main theorem and Theorems A and B of \cite{hp15}.

Wise proved $\pi_1M$ is virtually compact special when $M$ is a hyperbolic manifold with boundary in \cite{w12a}. Agol, Groves, and Manning proved $\pi_1M$ is virtually compact special when $M$ is a closed hyperbolic 3-manifold in \cite{ago13} building on \cite{w12a} and \cite{bw12}. If $M$ is a spherical manifold then $\pi_1M$ is finite and hence virtually trivial implying it is virtually compact special. For $M$ a sol manifold, then Hagen-Przytycki observed $\pi_1M$ is not virtually compact special. Indeed, since $\pi_1M$ is solvable but not virtually Abelian, the solvable subgroup theorem (See e.g. Theorem 7.8 in part II of \cite{bh99}.) implies $\pi_1M$ cannot act properly on a $CAT(0)$ cube complex and therefore is not virtually compact special. Technically, we could also consider sol manifolds as graph manifold with a nontrivial JSJ decomposition. Hagen-Przytycki \cite{hp15} exclude this case when studying graph manifolds but their results still hold since sol manifolds are not chargeless.

For the remaining cases we study their Seifert fibered structure. These cases are also discussed in \cite{hp15}. The argument is straightforward so we include it here.  Suppose $M$ is a Seifert fibered space with infinite fundamental group. If $M$ is closed and has a vanishing Euler number then it has a finite cover which is a product $F \times S^1$ where $F$ is a surface. It is well-known that the fundamental group of a such a manifold is virtually compact special. The geometries corresponding to a vanishing Euler number are $\mathbb{E}^3$, $\mathbb{H}^2 \times \mathbb{R}$, and $S^2 \times \mathbb{R}$. (See e.g. Table 1 in \cite{afw13}.) If $M$ has nonempty boundary then it is virtually the product of a surface with boundary and a circle and admits one of the three geometries above. 

If $M$ is a closed Seifert fibered space with a non-vanishing Euler number, then $M$ does not have a finite cover which is the product of a surface and a circle. Theorem 6.12 in Part II of \cite{bh99} states that if $\pi_1M$ were to act properly by isometries on a $CAT(0)$ space then there would be a finite index subgroup with the fiber subgroup as a direct factor. This would then implies (by e.g. Theorem 3.9 of \cite{awf15}) that $M$ has a finite cover which is the product of a surface and a circle. It follows that $\pi_1M$ cannot properly on a $CAT(0)$ cube complex if $M$ has non-vanishing Euler number. This excludes closed 3-manifolds which admit a $\widetilde{SL}(2, \mathbb{R})$ or nil geometric structure.

\end{proof}

\bibliographystyle{alpha}
\bibliography{TidmoreBib}

\end{document}